\documentclass[11pt]{article}

\usepackage[english]{babel}
\usepackage{geometry}
\usepackage[utf8]{inputenc}
\usepackage{mathtools,amsfonts,amsthm,mathrsfs,amssymb,stmaryrd}
\usepackage{textcomp}
\usepackage{bookmark}
\usepackage{microtype}
\usepackage{enumitem}
\usepackage{ulem}
\usepackage{todonotes}
\usepackage{fullpage}
\usepackage{verbatim}
\usepackage{cleveref}
\usepackage{authblk,enumitem}

\usepackage{tikz}
\usetikzlibrary{decorations.pathreplacing,decorations.markings}
\usepackage{standalone}

\newtheorem{thm}{Theorem}[section]
\newtheorem{lemma}[thm]{Lemma}
\newtheorem{proposition}[thm]{Proposition}
\newtheorem{corollary}[thm]{Corollary}
\newtheorem{conj}{Conjecture}
\newtheorem{problem}{Problem}

\theoremstyle{definition}
\newtheorem*{definition*}{Definition}

\newcommand{\eps}{\varepsilon}

\newcommand{\aG}{\vec{G}}

\newcommand{\E}{\mathcal{E}}

\newcommand{\aell}{\vec{\ell}}

\newcommand{\dich}{\vec{\chi}}
\newcommand{\bid}{\overset{{}_{\shortleftarrow \!\!\! \shortrightarrow}}}

\newcommand{\sG}{\mathscr{G}}

\newcommand{\pr}[1]{\mathbb{P}\left[#1\right]}
\newcommand{\esp}[1]{\mathbb{E}\left[#1\right]}
\newcommand{\sset}[2]{\left\{#1 : #2 \right \}}
\newcommand{\pth}[1]{\left(#1\right )}

\newcommand{\bigO}[1]{O\pth{#1}}
\newcommand{\floor}[1]{\left\lfloor #1 \right\rfloor}
\newcommand{\ceil}[1]{\left\lceil #1 \right\rceil}
\newcommand{\ind}[1]{[#1]}

\DeclareMathOperator{\TT}{TT}
\DeclareMathOperator{\suc}{succ}
\DeclareMathOperator{\pred}{pred}

\newcommand{\PA}[1]{{\color{blue}{\bf PA:} #1}}

\newcommand{\FP}[1]{{\color{purple}{\bf FP:} #1}}

\title{Minimum acyclic number and maximum dichromatic number of oriented triangle-free graphs of a given order}

\author[1]{Pierre Aboulker} 
\author[2]{Fr\'ed\'eric Havet}
\author[3]{Fran\c{c}ois Pirot}
\author[1]{Juliette Schabanel}

\affil[1]{DIENS, École normale supérieure, CNRS, PSL University, Paris, France}
\affil[2]{Universit\'e C\^ote d'Azur, CNRS, Inria, I3S, Sophia Antipolis, France}
\affil[3]{Université Paris-Saclay, Gif sur Yvette, France.}

\begin{document}

\maketitle

\begin{abstract}
Let $D$ be a digraph.
Its acyclic number $\vec{\alpha}(D)$ is the maximum order of an acyclic induced subdigraph and its dichromatic number $\vec{\chi}(D)$ is the least integer $k$ such that $V(D)$ can be partitioned into $k$ subsets inducing acyclic subdigraphs.
We study ${\vec a}(n)$ and $\vec t(n)$ which are the minimum of $\vec\alpha(D)$ and the maximum of 
$\vec{\chi}(D)$, respectively, over all oriented triangle-free graphs of order $n$. For every $\eps>0$ and $n$ large enough, we show
$(1/\sqrt{2} - \eps) \sqrt{n\log n} \leq \vec{a}(n) \leq \frac{107}{8} \sqrt n \log n$ and 
$\frac{8}{107} \sqrt n/\log n \leq \vec{t}(n) \leq  (\sqrt 2 + \eps) \sqrt{n/\log n}$.
We also construct an oriented triangle-free graph on 25 vertices with dichromatic number~3, and show that every oriented triangle-free graph of order at most 17 has dichromatic number at most 2.
\end{abstract}






\section{Introduction}

The \emph{order} of a graph or digraph is its number of vertices.
An \emph{independent set} in a graph is a set of pairwise non-adjacent vertices and a \emph{clique} is a 
set of pairwise adjacent vertices.
The \emph{independence number} of a graph $G$, denoted by $\alpha(G)$, is the size of a maximum independent set in $G$.
A \emph{triangle} is a clique of size $3$ and a \emph{triangle-free} graph is a graph with no triangle. 
The \emph{Ramsey number} $R(s,t)$ is the least $n$ such that every graph of order $n$ contains either a
clique of size $s$ or an independent set of size $t$. 
The \emph{inverse Ramsey number} $Q(s,n)$, is the minimum independence number of a graph of order $n$ with no clique of size $s$.
In other words, $R(s,t)=n$ if and only if  
$Q(s,n)=t$.
In particular, $R(3,t)$ is the minimum $n$ such that $\alpha(G) \geq t$ for every triangle-free graph of order $n$, and
$Q(3,n)$ the minimum of $\alpha(G)$ over all triangle-free graphs of order $n$. 
Ajtai, Koml\'os, and Szemer\'edi~\cite{AKS80} proved that $R(3,t)=O(t^2/\log t)$ and $Q(3,n) = \Omega(\sqrt{n\log n})$.
This was further tighten by Shearer~\cite{She83}, Fiz Pontiveros, Griffiths, and Morris~\cite{FGM20}, and Bohman and Keevash~\cite{BK21} :
\begin{eqnarray*}
\left(\frac{1}{4} - o(1)\right) \frac{t^2}{\log t}  & \leq ~R(3,t)~ \leq & (1+o(1)) \frac{t^2}{\log t} \\ 
\left(\frac{1}{\sqrt{2}} - o(1)\right) \sqrt{n\log n} & \leq ~Q(3,n)~ \leq & (\sqrt{2}+o(1)) \sqrt{n\log n}
\end{eqnarray*}

For some positive integer $k$, we use $[k]$ to denote the set $\{1,\ldots,k\}$. 
A \emph{$k$-colouring} of a graph (resp. digraph) $G$ is a function $\phi:V(G)\rightarrow [k]$. It is {\it proper} if $\phi(u)\neq \phi(v)$ for every edge (resp. arc) $uv \in E(G)$. We say that $G$ is {\it $k$-colourable} if it admits a proper $k$-colouring. 
The \emph{chromatic number} $\chi(G)$ of a graph or digraph $G$ is the least integer $k$ such that  $G$
is $k$-colourable.

In 1967, Erd\H{o}s~\cite{Erd67} asks for the greatest chromatic number $t(n)$ of a triangle-free graph of order $n$. 
Iteratively pulling out the largest independent set and using Ajtai-Koml\'os-Szemer\'edi bound on $Q(3,\cdot)$, Erd\H{o}s and Hajnal~\cite{ErHa85} proved
$t(n) = O(\sqrt{n/\log n})$.
Since then, several papers~\cite{ErHa85,She83,Kim95,FGM20,BK21,DaIl22+} improved the lower and upper bounds on $t(n)$. 
The best upper bound so far has been established by Davies and Illingworth~\cite{DaIl22+}:
\begin{equation} \label{eq:upper_bound_nono}
    t(n) \leq \pth{2\sqrt 2 + o(1)}\sqrt{\frac{n}{\log n}}
\end{equation}
The best lower bound was obtained by Bohman and Keevash~\cite{BK21} by analysing the triangle-free process.
\begin{equation}
    t(n) \geq \pth{\frac{1}{\sqrt{2}} - o(1)}\sqrt{\frac{n}{\log n}}
\end{equation}
So there is a factor
of $4$ between these upper and lower bounds on $t(n)$.

\medskip

The aim of this paper is to investigate similar questions for \emph{oriented triangle-free graphs}, that are oriented graphs with no triangle in their underlying graph. 

\medskip

A set of vertices is \emph{acyclic} in a digraph if it induces an acyclic subdigraph.
The \emph{acyclic number} 
of a digraph $D$, denoted by $\vec \alpha(D)$, is
the maximum size of an acyclic set in $D$, that is
the order of the largest acyclic induced subdigraph of $D$. It can be seen as a generalization of the independence number of an undirected graph.
Indeed, denoting by $\bid{G}$ the bidirected graph associated to $G$ (which is the digraph obtained from $G$ by replacing each edge by two arcs in opposite direction between its end-vertices), we trivially have
$\alpha(G) = \vec \alpha(\bid{G})$.

In 1964, Erd\H{o}s and Moser\cite{ErMo64} asked for a directed Ramsey-type problem.

\begin{problem}[Erd\H{o}s and Moser\cite{ErMo64}]
What is the least integer $\TT(k)$ such that every oriented graph on $\TT(k)$ vertices has acyclic number at least $k$ ?\\
What is $\vec \alpha(n) = \min \{
 \vec \alpha(D) \mid D \text{ is an oriented graph of order $n$}\}$ ?
 \end{problem}

By definition, $TT(k) = \min \vec \alpha^{-1}(k)$.
Note that $TT(k)$ is the least integer $n$ such that every tournament of order $n$ contains a transitive subtournament on $k$ vertices.
Erd\H{o}s and Moser\cite{ErMo64} proved 
\begin{eqnarray}
2^{(k-1)/2} & < ~\TT(k)~ \leq & 2^{k-1}\\ 
\log_2 n +1 & \leq ~\vec \alpha(n)~ \leq &  2 \log_2 n + 2
 \end{eqnarray}
Since this seminal paper, only little progress has been made on the above problem.
Using the Local Lemma, Nagy~\cite{Nagy14} slightly improved on the lower bound on $TT(k)$ by showing that, for every $C<1$,  $TT(k) > C\cdot 2^{(k+1)/2}$ when $k$ is large enough.
$TT(k)$ has been determined for small values of $k$.
Clearly $\TT(1) =1$, $\TT(2) =2$,
$\TT(3) =4$ (because of the directed $3$-cycle) and $\TT(4) =8$ (because of the Paley tournament on $7$ vertices).
Reid and Parker~\cite{RePa70} showed $\TT(5) =14$ and $\TT(6) =28$. Sanchez-Flores~\cite{San98} proved $\TT(7) =54$. 
This result implies $\TT(k)\leq 54 \times 2^{k-7}$ for all $k\geq 7$.

\medskip

A \emph{$k$-dicolouring} of a digraph $D$ is a function $\phi:V(D)\rightarrow [k]$ such that $D\ind{\phi^{-1}(i)}$ is acyclic for every $i \in [k]$. We say that $D$ is \emph{$k$-dicolourable} if it admits a $k$-dicolouring.
The \emph{dichromatic number} $\dich(D)$ of a digraph  $D$ is the least integer $k$ such that  $D$ is $k$-dicolourable. 
A digraph is \emph{$k$-dichromatic} if $\dich(D)=k$.

Analogously to Erd\H{o}s' question,  Neumann-Lara~\cite{Neu94} asked for the greatest dichromatic number $\vec \chi(n)$ of an oriented graph of order $n$.

\begin{problem}[Neumann-Lara~\cite{Neu94}]
What is $\vec \chi(n) = \max \{
 \vec \chi(D) \mid D \text{ is an oriented graph of order $n$}\}$?
 \end{problem}

Observe first that $\vec \chi(n)$ is attained by a tournament. 
Moreover, for every oriented graph $D$ of order $n$, $\vec\chi(D) \geq \frac{n}{\vec \alpha(D)}$ and a dicolouring of $D$ may be obtained by iterately pulling out the largest
acyclic set.
Hence, using the above bonds on $\vec \alpha(n)$, one easily gets
\begin{equation}
  \frac{1}{2} \frac{n}{ \log n +1} \leq \dich(n) \leq (1 + o(1)) \frac{n}{\log n}
\end{equation}
So there is a factor
of $2$ between these upper and lower bounds on $\dich(n)$. However,  for small values of $n$, $\dich(n)$ is precisely known. 
The smallest $2$-dichromatic tournament is the directed $3$-cycle. 
For every prime integer $n$ of the form $4k + 3$, the Paley tournament of order $n$ is the tournament $P_n$
whose vertex set is $\{0,\dots , n-1\}$ and containing the arc $ij$ if and only if $i-j$ is a square modulo
$n$.
Neumann-Lara~\cite{Neu94} proved that the smallest $3$-dichromatic tournament has
order 7 and that there exist four such tournaments, including $P_7$. He also proved that the
smallest $4$-dichromatic tournament has order 11, is unique and is $P_{11}$. The Paley tournament $P_{19}$ is
4-dicolourable but Neumann-Lara~\cite{NL00} showed a 5-dichromatic tournament on 19 vertices. Recently, Bellitto et al.~\cite{bellitto2023smallest} proved that it is actually a smallest one : there is no 5-dichromatic oriented graph of order less than 19.
Hence 
 \[\dich(n)\:=\left\{
    \begin{array}{ll}
      1,&\text{if $1\leq n \leq  2$,}\\
      2,&\text{if $3\leq n \leq  6$,}\\
      3,&\text{if $7\leq n \leq  10$,}\\
      4,&\text{if $11\leq n \leq  18$,}\\
      5,&\text{if $n =  19$.}\\     
    \end{array}\right.\]

\medskip

 Similar questions may be asked for subclasses of oriented graphs, and in particular for $H$-free oriented graphs for a given oriented graph $H$.
 (A digraph is \emph{$H$-free} if it does not contain $H$ as a (not necessarily induced)  subdigraph.
  In an unpublished work, Harutyunyan and McDiarmid proposed the following conjecture. 

\begin{conj}[Harutyunyan and McDiarmid \cite{HCunpublished}]
For every oriented graph $H$, there is $\epsilon > 0$ such that every $H$-free oriented graph $D$ of order $n$ satisfies $\vec \alpha(D) \geq n^{\epsilon}$ and  $\dich(D) \leq n^{1-\epsilon}$. 
\end{conj}

This conjecture is open even when $H$ is the directed cycle of length $3$.  In fact, it is a strengthening of the following conjecture, which is equivalent to the celebrated Erd\H{o}s-Hajnal conjecture.

\begin{conj}[Alon, Pachs, Solymosi~\cite{APS01}]
For every tournament $H$, there exists $\epsilon > 0$ such that every $H$-free tournament $T$ of order $n$ satisfies $\vec \alpha(T) \geq n^\epsilon$.
\end{conj}
 This conjecture is known to hold for a few
types of tournaments $H$ \cite{BCC15,BCC19}, but is still wide open in general.

\subsection*{Our results}
In this paper, we study the acyclic and dichromatic numbers of oriented triangle-free graphs.

More precisely, we give some bounds on 
\begin{enumerate}[label=(\roman*)]
\item ${\vec a}(n)$, the minimum of $\vec\alpha(D)$ over all oriented triangle-free graphs of order $n$, and  
\item $\vec t(n)$, the maximum dichromatic number of an oriented triangle-free graph of order $n$. 
\end{enumerate}

By definition, ${\vec a}(n) \geq Q(3,n)$, and so
${\vec a}(n) \geq \left(\frac{1}{\sqrt{2}} - o(1)\right) \sqrt{n\log n}$.
In Subsection~\ref{subsec:lower}, considering a well-chosen orientation of $G(n,p)$ with $p=c_0/\sqrt{n}$, we prove that $\vec{a}(n) \leq \frac{107}{8} \sqrt n \log n$ for $n$ sufficiently large.
Hence we have, for every $\eps>0$ and $n$ large enough,
\begin{equation}\label{eq:a}
\left(\frac{1}{\sqrt{2}} - \eps \right) \sqrt{n\log n} \leq \vec{a}(n) \leq \frac{107}{8} \sqrt n \log n.
\end{equation}

We believe that a similar analysis can be performed on the $n$-vertex triangle-free process $G_\triangle$ (see \cite{Boh09}), to prove that an orientation of $G_\triangle$ has acyclic number $\bigO{\sqrt{n\log n}}$. We thus conjecture: 

\begin{conj}\label{conj:a} $\vec{a}(n) = \Theta(\sqrt{n\log n})$.  
\end{conj}

The above upper bound on  $\vec{a}(n)$ 
immediately yields $\vec{t}(n) \geq \frac{8}{107}\frac{\sqrt n}{\log n}$.
Moreover, $\vec{t}(n) \leq t(n)$ and so $\vec t(n) \leq \pth{2\sqrt 2 + o(1)}\sqrt{\frac{n}{\log n}}$. In Subsection~\ref{sec:upper_bound}, we improve this bound by a factor of $2$.  
 We thus have, for every $\eps>0$ and $n$ large enough,
\begin{equation}\label{eq:main}
\frac{8}{107} \frac{\sqrt n}{\log n} \leq \vec{t}(n) \leq  \pth{\sqrt 2 + o(1)} \sqrt{\frac{n}{\log n}}.
\end{equation}

Conjecture~\ref{conj:a} implies the following.

\begin{conj}
    $\vec t(n) = \Theta \sqrt{\frac{n}{\log n}}$. 
\end{conj}

\medskip

Determining $\vec t(n)$ is equivalent to determining $\vec m(k)$ the minimum order of a $k$-dichromatic oriented triangle-free graph, because
$\vec m(k) = \min \vec t^{-1}(k)$. In Section~\ref{sec:det},
 we present a deterministic way of building small oriented triangle-free graphs with dichromatic number at least $k$ from an undirected triangle-free graph with chromatic number at least $k$. We then use this construction to build a $3$-dichromatic oriented triangle-free graph on $25$ vertices. Then we prove that all oriented triangle-free graphs on at most $17$ vertices are $2$-dicolourable. Hence 
\begin{equation}
    18 \leq \vec m(3) \leq 25
\end{equation}
Using our construction, we also obtain that $\vec m(4) \leq 209$.

\medskip

In order to prove the upper bound on $\vec a(5)$ given in~\eqref{eq:a}, we first show in Lemma~\ref{lem:dsparse-acyclic}, that every graph  
$G$ of order $n$, admits an orientation such that each of its acyclic sets is $d$-sparse with $d\coloneqq 2\log_2 n + 1$ in $G$. A set of vertices $X\subseteq V(G)$ is \emph{$d$-sparse} in $G$ if the average degree of $G[X]$ is at most $d$.
 This lemma also yields interesting bounds 
 between the (list) chromatic number 
 of graph $G$ and $\dich(G)$, the maximum dichromatic number over all orientations of $G$. 
 In Theorems ~\ref{thm:ratio-dichromatic} and~\ref{thm:dichromatic-list}, we prove that, for every graph $G$ of order $n$, one has
\begin{align}
 \chi(G) &\le 2\dich(G)\pth{1+\floor{\log_2 n}} \label{eq:chi},\quad \mbox{and}  \\
 \chi_{\ell}(G) &\le 6\dich(G)\pth{1+\floor{\log_2 n}}\label{eq:chi_l}.
\end{align}

\section{Preliminary results}\label{sec:def_prel}


\subsection{Probabilistic tools}\label{subsec:proba_tools}

\begin{lemma}[c.f. \cite{KaMc10}]
    \label{lem:binomial}
    Let $X$ be a random variable distributed according to the binomial distribution $\mathcal{B}(n,p)$, for some integer $n$ and some $p\in (0,1)$. For every $x\in (0,1)$, we let $\Lambda^*(x) \coloneqq x \ln \frac{x}{p} + (1-x) \ln \frac{1-x}{1-p}$.
    Then, for every integer $k < n$, one has
    \[ \pr{X \le k} \le e^{-n\, \Lambda^*(k/n)}.\]
\end{lemma}

\begin{lemma}[Chernoff's bound]
    \label{lem:chernoff}
    Let $X$ be a sum of i.i.d. 
    $(0,1)$-valued variables. Then, for every $\delta \in (0,1)$,
    \[ \pr{X < (1-\delta) \esp{X}} < \pth{\frac{e^{-\delta}}{(1-\delta)^{1-\delta}}}^{\esp{X}}\]
\end{lemma}

\begin{lemma}[Lov\'asz Local Lemma]
\label{lem:LLL}
Let $\{E_i\}_{i \in [n]}$ be a set of random (bad) events, with dependence graph $\Gamma$ (i.e. each event $E_i$ is mutually independent from every random event that is not in its neighbourhood in $\Gamma$).
If there exist $y_1, \ldots, y_n > 0$ such that $\pr{E_i} < 1/y_i$, and moreover 
\[ \ln y_i > \sum_{E_j \in N_\Gamma(E_i)} y_j \pr{E_j}, \]
for every $i \in [n]$, then $\pr{\bigwedge_{i\in [n]} \overline{E_i}} >0$.
\end{lemma}


\subsection{A useful orientation}\label{subsec:useful_or}

The following lemma is the main tool used to prove the lower bound in~\eqref{eq:main}. 

\begin{lemma}
\label{lem:dsparse-acyclic}
    Let $G$ be a graph of order $n$, and fix $d\coloneqq 2\log_2 n + 1$.
    Then there exists an orientation $\aG$ of $G$ such that every acyclic set of vertices in $\aG$ is $d$-sparse in $G$.
\end{lemma}

\begin{proof}
Let $\aG$ be a uniformly random orientation of $G$. Let $X$ be any subset of vertices of $G$ inducing a subgraph of average degree at least $d \coloneqq 2\log_2 n +1$. We denote $s\coloneqq |X|$; there are at least $ds/2$ edges in $G[X]$.
For each acyclic orientation of $G[X]$, there is an ordering $(x_1, \dots , x_s)$ of $X$ that $x_ix_j$ is an arc if and only if $x_ix_j \in E(G)$ and $i<j$.
Hence the number of acylic orientations of $G[X]$ is at most the number of orderings of $X$, so at most $s!$. Since there are exactly $2^{|E(G[X])|} \ge 2^{ds/2}$ orientations of $G[X]$, we have
\[ \pr{\mbox{$X$ is acyclic in $\aG$}} \le \frac{s!}{2^{ds/2}}.\]
On the other hand, there are $\binom{n}{s}$ possible choices for a subset $X$ of $s$ vertices, so the probability that there exists an acyclic induced subdigraph of $\aG$ of density more than $d$ is at most 
\[ \sum_{s\ge 1} \binom{n}{s} \frac{s!}{2^{ds/2}} \le \sum_{s\ge 1} \left(\frac{n}{2^{d/2}}\right)^s \le \sum_{s\ge 1} \left(\frac{1}{2}\right)^s < 1.\]

We conclude that there exists an orientation $\aG$ of $G$ such that every acyclic set of $\aG$ is $d$-sparse, as desired.
\end{proof}

Lemma~\ref{lem:dsparse-acyclic} has interesting consequences when we seek for a lower bound on the dichromatic number of a graph $G$ in terms of its chromatic parameters.

\begin{thm}
\label{thm:ratio-dichromatic}
For every graph $G$ of order $n$, one has
\[ \chi(G) \le 2\dich(G)\pth{1+\floor{\log_2 n}}.\]
\end{thm}

\begin{proof}
Let $\aG$ be the orientation of $G$ given by Lemma~\ref{lem:dsparse-acyclic}.
Then every acyclic subdigraph of $\aG$ is $\floor{d}$-degenerate, and hence has a proper $(\floor{d}+1)$-colouring, where $d=2\log_2n+1$. Let $c$ be a dicolouring of $\aG$ with $\dich(\aG)\le \dich(G)$ colours. We can find a proper colouring of $G$ by colouring independently each colour class of $c$ using at most $\floor{d}+1$ (new) colours. Hence
\[ \chi(G) \le \dich(G)(\floor{d}+1) = 2\dich(G)\pth{1+\floor{\log_2 n}}.\]
\end{proof}

For a given graph $G$, we let $\chi_\ell(G)$ be the \emph{list-chromatic number} of $G$ (sometimes called the \emph{choosability} of $G$). We recall that this is the minimum integer $k$ such that, for every $k$-list-assignment $L\colon V(G) \to \binom{\mathbb{N}}{k}$, there exists a proper $L$-colouring of $G$, that is a proper colouring $c\colon V(G) \to \mathbb{N}$ that satisfies $c(v) \in L(v)$ for every vertex $v\in V(G)$.
A greedy algorithm shows that, for every integer $d\ge 0$, every $d$-degenerate graph $G$ is $(d+1)$-list-colourable, i.e. $\chi_\ell(G) \le d+1$.

\begin{thm}
    \label{thm:dichromatic-list}
    For every graph $G$ of order $n$, one has
\[ \chi_{\ell}(G) \le 6\dich(G)\pth{1+\floor{\log_2 n}}.\]
\end{thm}

\begin{proof}
    As in the proof of Theorem~\ref{thm:ratio-dichromatic}, there exists an orientation $\aG$ of $G$ such that each acyclic subdigraph of $\aG$ is $\floor{d}$-degenerate, and hence $(\floor{d}+1)$-list-colourable, where $d=2\log_2n+1$. We write $k_0\coloneqq \floor{d}+1 \ge d$.
    Let $q=\dich(\aG)\le \dich(G)$, and let $\phi$ be a $q$-dicolouring of $\aG$.
    
    Let $k= 3qk_0 \le 6\dich(G)(1+\floor{\log_2 n})$, and let $L\colon V(G) \to \binom{\mathbb{N}}{k}$ be any $k$-list-assignment of $G$. We let $X\coloneqq \bigcup_{v\in V(G)} L(v)$ be the set of colours covered by $L$.  
    For every colour $x\in X$, we assign a label $\sigma(x)$ chosen uniformly at random from $[q]$.
    For every vertex $v\in V(G)$, we let $L_\sigma(v)\coloneqq \sset{x\in L(v)}{\sigma(x)=\phi(v)}$ be the list of colours in $L(v)$ that have been assigned $\phi(v)$ as their label. Then the size of $L_\sigma(v)$ is the sum of $k$ independent Bernoulli variables of parameter $1/q$. By Lemma~\ref{lem:chernoff}, we have 
    \begin{align*}
        \pr{|L_\sigma(v)|< k_0} &< \pth{\frac{e^{-2/3}}{(1/3)^{1/3}}}^{k/q} < e^{-0.9 k_0} \le e^{-0.9 d} < 1/n.
    \end{align*}
    We conclude with a union-bound that with non-zero probability we have $|L_\sigma(v)|\ge k_0$ for every vertex $v\in V(G)$. Each colour class of $\phi$ is therefore $L_\sigma$-colourable, and since by construction the colours in $L_\sigma$ are disjoint between different colour classes of $\phi$, we conclude that $G$ is $L_\sigma$-colourable. We finish the proof by noting that any proper $L_\sigma$-colouring is also in particular a proper $L$-colouring.
\end{proof}


\section{Bounds on $\vec{a}(n)$ and $\vec{t}(n)$}


In this section, we establish the bounds in~\eqref{eq:a} and~\eqref{eq:main}. 




\subsection{Upper bound on $\vec{a}(n)$}\label{subsec:lower}

\begin{thm}\label{thm:lower_bound}
    If $n$ is sufficiently large, then there exists an oriented triangle-free graph $G$ of order $n$ such that $\vec \alpha(G) \leq \frac{107}{8} \sqrt n \log n$. So 
    $\vec a(n) \leq \frac{107}{8} \sqrt n \log n$.
\end{thm}

\begin{proof}
    Let $c_0>0$ be a fixed constant, whose specific value will be determined later in the proof.
    Let $G$ be a graph drawn from $G(n,p)$, with $p=c_0/\sqrt{n}$.
    
    By Lemma~\ref{lem:dsparse-acyclic}, there exists an orientation $\aG$ of $G$ such that every set that induces an acyclic subdigraph of $\aG$ is $d$-sparse, for $d\coloneqq 2\log_2 n + 1$.
    To finish the proof, we show that, with non-zero probability, $G$ has no triangle and there is no $d$-sparse set of size 
    $k\coloneqq c_1\, \frac{d}{p} + 1$ in $G$, for suitable values of $c_0,c_1$ (given at the end of the proof). 

    Let $X\subseteq V(G)$ be a subset of $k$ vertices of $G$. The number of edges in $G[X]$ follows the binomial distribution $\mathcal{B}(\binom{k}{2},p)$, and we note that $X$ is $d$-sparse if and only if $G[X]$ has at most $dk/2$ edges. 
    For $x\in (0,1)$, we recall that $\Lambda^*(x) = x \ln \frac{x}{p} + (1-x) \ln \frac{1-x}{1-p}$. By Lemma~\ref{lem:binomial}, we have
    \begin{align*}
        \pr{\mbox{$X$ is $d$-sparse}} &\le e^{-\binom{k}{2} \Lambda^*\pth{\frac{d}{k-1}}} \\
        &\le e^{-\frac{(k-1)^2}{2}\pth{\frac{p}{c_1}\ln \frac{1}{c_1} + \pth{1-\frac{p}{c_1}}\pth{\ln (1-p/c_1) - \ln (1-p)}}} \\
        &\le e^{-\frac{c_1^2d^2}{2 p^2}\pth{\frac{p}{c_1}\ln \frac{1}{c_1} + \pth{1-\frac{p}{c_1}}\pth{-\frac{p}{c_1}+p+O(p^2)}}} \\
        &\le e^{-\frac{c_1^2d^2}{2p^2}\pth{\frac{c_1-1-\ln c_1}{c_1}\,p + O(p^2)}} = e^{-c_1(c_1-1-\ln c_1)\frac{d^2}{2p} + O(d^2)}.
    \end{align*}

    For every set $T \in \binom{V(G)}{3}$, we let $A_T$ be the random event that $G[T]$ is a triangle. We have $\pr{A_T}=p^3$.
    For every set $X \in \binom{V(G)}{k}$, we let $B_X$ be the random event that $X$ is $d$-sparse. 
    By the above, we have $\pr{B_X} \le e^{-c_1(c_1-1-\ln c_1)d^2/2p + O(d^2)}$. 
    Let $\mathcal A = \{A_T \mid T \in \binom{V(G)}{3}\}$, $\mathcal B = \{B_X \mid X \in \binom{V(G)}{k}\}$ and $\E \coloneqq \mathcal A \cup \mathcal B$. We wish to apply Lemma~\ref{lem:LLL} in order to show that, with non-zero probability, no random event in $\E$ occurs. 
    The rest of the proof is similar to that in \cite{Spe77}.
    Let $\Gamma$ be the dependence graph of $\E$. Two events in $\E$ are adjacent in $\Gamma$ if and only if they concern sets that intersect on at least $2$ vertices (since otherwise they depend on disjoint sets of edges). For every $\mathcal X,\mathcal Y \in \{\mathcal A,\mathcal B\}$, we let $N_{\mathcal X \mathcal Y}$ be the maximum number of nodes of $\mathcal Y$ adjacent to a given node of $\mathcal X$ in $\Gamma$
    . We have 
    \begin{align*}
        N_{\mathcal A\mathcal B}, N_{\mathcal B\mathcal B} &\le |\mathcal B| = \binom{n}{k} < \pth{\frac{ne}{k}}^k;\\
        N_{\mathcal A\mathcal A} &= 3(n-3)<3n\; ;\\
        N_{\mathcal B\mathcal A} &= \binom{k}{2}(n-k)+\binom{k}{3} < \frac{(k-1)^2 n}{2} = \frac{1}{2}\pth{\frac{c_1dn}{c_0}}^2 \quad \mbox{for $n$ large enough.}        
    \end{align*}
    
    For Lemma~\ref{lem:LLL} to apply, there remains to find $y,z>0$ such that, for every $T\in \binom{V(G)}{3}$ and $X \in \binom{V(G)}{k}$,
    \begin{equation}
    \label{eq:LLL}
    \begin{split}
        \pr{A_T} &< 1/y, \quad \pr{B_X} < 1/z \; ;\\
        \ln y &> y \pr{A_T}N_{\mathcal A\mathcal A} + z \pr{B_X} N_{\mathcal A\mathcal B} \; ;\\
        \ln z &> y \pr{A_T}N_{\mathcal B\mathcal A} + z\pr{B_X} N_{\mathcal B\mathcal B}.
        \end{split}
    \end{equation}

\noindent
    Let $y=1+\eps$ for some $\eps>0$ small enough, and let $z = e^{c_2 d^2/2p}$ for some $0<c_2 < c_1(c_1-1-\ln c_1)$. With these values, we have 
    \begin{align*}
        z\pr{B_S} \binom{k}{n} &< \exp\pth{\frac{d^2}{2p}(c_2-c_1(c_1-1-\ln c_1)) + k \ln \frac{ne}{k} + O(d^2)} \hspace{-1000pt} \\
        &< \exp\pth{\frac{d^2}{2p}(c_2-c_1(c_1-1-\ln c_1)) + \frac{d^2}{2p}(1+o(1))} \hspace{-1000pt} \\
        &\underset{n\to \infty}{\to} 0 & \mbox{if $c_1(c_1-1-\ln c_1)>1+c_2$;}\\
        y\pr{A_T}N_{\mathcal B \mathcal A}&< (1+\eps) \frac{p^3}{2}\pth{\frac{c_1dn}{c_0}}^2 = (1+\eps) \frac{c_0 c_1^2}{2} d^2 \sqrt{n} \hspace{-1000pt} \\
        &< (1+\eps)(c_0c_1)^2 \frac{d^2}{2p} < \ln z - \eps & \mbox{if $c_2 > (1+\eps)(c_0c_1)^2$ and $n$ is large enough.}
    \end{align*}
    It is straightforward that \eqref{eq:LLL} holds whenever the two conditions on $c_0,c_1,c_2$ above hold, and $n$ is large enough. In order to minimise the value of $k$, we may fix $c_0=0.513$, $c_1=3.43$, and $c_2=3.1$.
    \end{proof}





\subsection{Upper bound on $\vec t(n)$} \label{sec:upper_bound}

The goal of this section is to prove the upper bound of~\eqref{eq:main}, see Theorem~\ref{thm:dich-ramsey}.  

\medskip 

Given a digraph $D$, and a total order $\prec$ on $V(D)$, we denote $D[\prec]$ the subdigraph induced by the arcs $(v,u) \in A(D)$ such that $u\prec v$. We may disregard the orientation of the edges in $D[\prec]$, in which case this is usually called the \emph{backedge graph} associated with $D$ and $\prec$. It is straightforward that every independent set $I$ of $D[\prec]$ induces an acyclic digraph $D[I]$.
A well-known consequence is that $\dich(D) \le \chi(D[\prec])$.
To the best of our knowledge, the possible tightness of that bound has never been discussed in the literature. We prove that there always exists $\prec$ such that $\dich(D)=\chi(D[\prec])$.

\begin{thm}
\label{thm:order-colouring}
    For every digraph $D$, one has
    \[ \dich(D) = \min \, \Big \{ \chi(D[\prec]) : \mbox{$\prec$ is a total order on $V(D)$} \Big\}.\]
\end{thm}

\begin{proof}
    We have already shown that $\dich(D) \le \chi(D[\prec])$ for every total order $\prec$ on $V(D)$.
    Let us prove the other direction.
    Let $k\coloneqq \dich(D)$, and let $\phi$ be a $k$-dicolouring of $D$. 
    Let $C_1, \ldots, C_k$ be the colour classes of $\phi$. For every $i\in [k]$, $D[C_i]$ is acyclic, and so induces a partial order on $C_i$. Let $\prec_i$ be a total order of $C_i$ that extends this partial order.
    For every $u,v \in V(D)$, we let $u\prec v \equiv (\phi(u) < \phi(v)) \vee (v \prec_{\phi(v)} u)$.
    It is straightforward that each $C_i$ is an independent set in $D[\prec]$, and so $\phi$ is a proper $k$-colouring of $D[\prec]$. This proves that $\chi(D[\prec]) \le \dich(D)$, as desired.
\end{proof}

Theorem~\ref{thm:order-colouring} reduces the problem of finding a minimum dicolouring of a digraph $D$ to finding a total order $\prec$ on $V(D)$ that minimises $\chi(D[\prec])$. Hence, any order $\prec$ that makes $D[\prec]$ sparse in a sense that lets us bound its chromatic number efficiently is of interest.
In the following, we prove that we can halve the maximum degree of a digraph with a well-chosen order.

\begin{lemma}
\label{lem:order-half-degree}
Let $D$ be a digraph of maximum total degree $\Delta$. Then there exists a total order $\prec$ on $V(D)$ such that $\deg_{D[\prec]}(v) \le \floor{\deg_D(v)/2} $ for every vertex $v\in V(D)$.
\end{lemma}

\begin{proof}
    Let $\prec$ be an order that minimises the number of edges in $D[\prec]$. Let us prove that $\prec$ has the desired property.
    For the sake of contradiction, assume that there is a vertex $v$ such that $\deg_{D[\prec]}(v) > \floor{\deg_D(v)/2}$. Let $\pred(v)$ be the set of neighbours $u$ of $v$ in $D$ such that $u\prec v$, and let $\suc(v) = N_D(v) \setminus \pred(v)$. 
    Then $v$ is adjacent in $D[\prec]$ with more than half the vertices in either $\suc(v)$ or $\pred(v)$, say in $\suc(v)$ without loss of generality. 
    Then, letting $\prec'$ be the order obtained from $\prec$ after making $v$ the largest element, we infer that $D[\prec']$ has strictly fewer edges than $D[\prec]$, a contradiction. 
\end{proof}

Given a class of graphs $\sG$, let $\chi(\sG) \coloneqq \max\, \{\chi(G) : G\in \sG\}$ and $\dich(\sG) \coloneqq \max\, \{\dich(G) : G\in \sG\}$.
For every integer $d$, we let $\sG_{\Delta\le d} \coloneqq \{G \in \sG : \Delta(G) \le d\}$ be the class of graphs of maximum degree at most $d$ in $\sG$.
Theorem~\ref{thm:order-colouring} applied together with Lemma~\ref{lem:order-half-degree} yields the following result as a corollary.

\begin{corollary}
\label{cor:dich_class}
    For every hereditary class of graphs $\sG$, and every integer $d$, one has
    \[ \dich(\sG_{\Delta \le d}) \le \chi(\sG_{\Delta \le \floor{d/2}}). \]
\end{corollary}

If we apply Corollary~\ref{cor:dich_class} to the class of triangle-free graphs, and use the bound given by the Johansson-Molloy theorem \cite{Mol19}, we obtain the following. 

\begin{thm}
    \label{thm:johansson-dicolouring}
    For every triangle-free graph $G$ of maximum degree $\Delta$, one has 
    \[\dich(G) \le (1+o(1)) \frac{\Delta}{2\ln \Delta} \quad \mbox{as $\Delta \to \infty$.}\]
\end{thm}


We are now ready to prove the main result of this section. 

\begin{thm}
\label{thm:dich-ramsey}
For every triangle-free graph $G$ of order $n$, one has \[\dich(G) \le (1+o(1)) \sqrt{\frac{2n}{\ln n}} \text{\quad as $n\to \infty$.}\]
\end{thm}

\begin{proof}
The following proof is similar in many aspects to that of \cite[Theorem~1]{DaIl22+}.
Let $0<\eps \le 1/2$, and let $d(n) \coloneqq \sqrt{2n\ln n}$ for all $n\ge 2$.
By Theorem~\ref{thm:johansson-dicolouring}, there exists $n_1$ such that, for every triangle-free digraph $D$ of maximum degree at most $d(n)$, one has $\dich(D) \le (1+\eps)\sqrt{2n/\ln n}$ if $n\ge n_1$.

Let $n_0 \coloneqq \max \,\{e^{\frac{1+\eps}{\eps}}, n_1\} \ge e^3$, and let $f(x) \coloneqq n_0 + (1+\eps) \sqrt{2x/\ln x}$ for all $x>1$.
Let us prove by induction on $n$ that, for every triangle-free graph $G$ of order $n$, one has $\dich(G) \le f(n)$.

If $n\le n_0$, then $n \le f(n)$ and the result holds trivially by assigning a distinct colour to every vertex.
We may now assume that $n\ge n_0$, which implies that $\frac{1}{\ln n} \le \frac{\eps}{1+\eps}$.
If $\Delta(G) \le d(n)$, this is a direct consequence of Theorem~\ref{thm:johansson-dicolouring} as explained above.
So may assume that $\Delta(G) \ge d(n)$. The neighbourhood of a vertex of maximum degree yields an independent set $I$ of size at least $d(n)$. Given a fixed orientation $D$ of $G$, any dicolouring of $D\setminus I$ may be extended to $D$ with an extra colour that we assign to the vertices in $I$. By the induction hypothesis, this shows that
\[ \dich(G) \le f(n-d(n)) + 1.\]
So the result holds if we can show that $f(n)-f(n-d(n))\ge 1$. The second derivative of $f$ is negative on the interval $[6,+\infty)$, so $f$ is concave on that interval. Since $n - d(n) \ge n_0 - d(n_0) > 6$, this implies that 
\begin{align*}
    f(n)-f(n-d(n)) &\ge d(n)f'(n) = \sqrt{2n\ln n} \cdot (1+\eps) \frac{\ln n-1}{(\ln n)^2}\sqrt{\frac{\ln n}{2n}} = (1+\eps)\pth{1-\frac{1}{\ln n}} \\
    &\ge (1+\eps)\pth{1-\frac{\eps}{1+\eps}} = 1,
\end{align*}
as desired.
\end{proof}



\section{A deterministic approach} \label{sec:det}


\subsection{Linear forests}

A \emph{directed linear forest} is a forest formed by a disjoint union of directed paths.
Given a directed graph $D$, we denote by $\aell(D)$ the maximum number of arcs  of a  directed linear forest of $D$.  
For every (undirected) graph $G$, we define $\aell(G) \coloneqq \min \; \{\aell(\aG) : \aG \mbox{ is an orientation of $G$}\}$.

\begin{lemma}
    \label{lem:LinearForest}
    For every graph $G$ of order $n$, $\aell(G) = n-\alpha(G)$.
\end{lemma}

\begin{proof}
Let $\aG$ be an orientation of $G$.
Every directed linear forest of $\aG$ may trivially be extended into a spanning directed linear forest with the same number of arcs, by adding all uncovered vertices as isolated trees of order $1$. Hence $\aell(\aG)$ is also the maximum number of arcs of a spanning directed linear forest of $\aG$.
Now, for a spanning directed linear forest $F$, the number of arcs of $F$ plus the number of paths of $F$ equals $n$. 
By the Gallai-Milgram Theorem \cite{GaMi60}, $\aG$ can be covered with at most $\alpha(G)$ directed paths, which together form a directed linear forest of size at least $n-\alpha(G)$. 
Hence $\aell(G) 
\geq n - \alpha(G)$.

To show the other direction, let $I$ be a maximum independent set of $G$, and let $D$ be an orientation of $G$ where each vertex $x\in I$ has in-degree $0$. Let $F$ be any directed linear forest of $D$; each vertex $x\in I$ belongs to a distinct directed path in $F$. So $F$ contains at least $\alpha(G)$ distinct paths, hence it has size at most $n-\alpha(G)$.
\end{proof}


\subsection{The construction}

Let $D$ be a digraph. The \emph{$n$-backward-blowup} $D^{\gets n}$ of $D$ is the digraph defined by 

\begin{eqnarray*}
V(D^{\gets n})  & = & \{(v,i) \mid v\in V(D), i\in [n]\};\\
A(D^{\gets n}) & = & F(D^{\gets n}) \cup B(D^{\gets n}) ,  \mbox{~~with} \\
F(D^{\gets n}) & = & \{(u,i) \to (v,i) \mid u\to v \in A(D), i\in [n]\} ,  \mbox{~and}\\
B(D^{\gets n}) & = & \{(v,i) \to (u,j) \mid u\to v \in A(D), i\in [n], j\in [n], \mbox{~and~} i\neq j \}
\end{eqnarray*}

$F(D^{\gets n})$ is the set of \emph{forward arcs} of $D^{\gets n}$ and $B(D^{\gets n})$ is the set of \emph{backward arcs} of $D^{\gets n}$.
For $v\in V(D)$, $\{(v,i) \mid i\in [n]\}$ is the \emph{pack} of $v$. 
Two vertices in a same pack are called  \emph{twins}. 

\begin{figure}[ht]
    \centering
\begin{tikzpicture}[>=stealth, very thick,decoration={
    markings, mark=at position 0.85 with {\arrow{>}}}] 

    \draw[orange, postaction={decorate}] (4, 1) -- (0, 0);
    \draw[orange, postaction={decorate}] (4, 2) -- (0, 0);
    \draw[orange, postaction={decorate}] (4, 3) -- (0, 0);

    \draw[orange, postaction={decorate}] (4, 0) -- (0, 1) ;
    \draw[orange, postaction={decorate}] (4, 2) -- (0, 1);
    \draw[orange, postaction={decorate}] (4, 3) -- (0, 1);

    \draw[orange, postaction={decorate}] (4, 0) -- (0, 2);
    \draw[orange, postaction={decorate}] (4, 1) -- (0, 2);
    \draw[orange, postaction={decorate}] (4, 3) -- (0, 2);

    \draw[orange, postaction={decorate}] (4, 0) -- (0, 3);
    \draw[orange, postaction={decorate}] (4, 1) -- (0, 3);
    \draw[orange, postaction={decorate}] (4, 2) -- (0, 3);
    
    \draw[cyan, postaction={decorate}] (0, 0) -- (4, 0);
    \draw[cyan, postaction={decorate}] (0, 1) -- (4, 1);
    \draw[cyan, postaction={decorate}] (0, 2) -- (4, 2);
    \draw[cyan, postaction={decorate}] (0, 3) -- (4, 3);

    \draw(0,0) node {$\bullet$} (0,1) node {$\bullet$} (0,2) node {$\bullet$} (0,3) node {$\bullet$} (4,0) node {$\bullet$} (4,1) node {$\bullet$} (4,2) node {$\bullet$} (4,3) node {$\bullet$};
    \draw (0,0) node[left] {$u_4$} (0,1) node[left] {$u_3$} (0,2) node[left] {$u_2$} (0,3) node[left] {$u_1$} 
    (4,0) node[right] {$v_4$} (4,1) node[right] {$v_3$} (4,2) node[right] {$v_2$} (4,3) node[right] {$v_1$}; 
\end{tikzpicture}
\end{figure}

\begin{lemma}
\label{lem:blowup}
Let $k$ be a positive integer, and let $G$ be a graph such that $\chi(G)>k$.
Then, for every orientation $\aG$ of $G$, the $(k\aell(\aG)+1)$-backward-blowup of $\aG$ is not $k$-dicolourable.
\end{lemma}

\begin{proof}
Let $\aG$ be an orientation of $G$ and $D$ be the $(k\aell(\aG)+1)$-backward-blowup of $\aG$, and let $\phi$ be any $k$-colouring of the vertices of $D$. Let us assume for the sake of contradiction that  $\phi$ induces no monochromatic directed cycle.

The forward arcs of $D$ form $k\aell(\aG)+1$ pairwise disjoint copies of $\vec{G}$. Since $\chi(G)>k$, there is at least one monochromatic arc in each of these copies; let us pick one for each copy arbitrarily and obtain a matching $M$ of $k\aell(\aG)+1$ monochromatic forward arcs. By the Pigeonhole Principle, we can find a set $N$ of $\aell(\aG)+1$ arcs of $M$ with the same colour. 

If two arcs of $N$ lie between the same pair of packs (say $u\to v$ and $u'\to v'$ where $(u,u')$ and $(v,v')$ are pairs of twins), they induce a monochromatic directed $4$-cycle (namely $u\to v \to u' \to v' \to u$), a contradiction.

If $N$ contains two arcs $u \to v$ and $w \to v'$ where $v$ and $v'$ are twins, then $u\to v \to w \to v' \to u$ is a monochromatic directed $4$-cycle in $D$, a contradiction.
Similarly, if $N$ contains two arcs $v \to u$ and $v' \to w$ where $v$ and $v'$ are twins, then $v\to u \to v' \to w \to v$ is a monochromatic directed $4$-cycle in $D$, a contradiction.

Henceforth, the graph $F$ induced by $N$ after contracting each pack into a single vertex has maximum in- and out-degree $1$. If $F$ contains a directed cycle, then there are $p$ arcs of $N$ that may be labelled $u_1 \to u'_2, u_2 \to u'_3, \ldots, u_{p-1} \to u'_p, u_p \to u'_1$ so that $u_i,u'_i$ are twins for every $i\in [p]$. We infer that $u_1 \to u_p \to \ldots \to u_2 \to u_1$ is a monochromatic directed cycle in $D$, a contradiction. 
We conclude that $F$ is a directed linear forest. Since $F$ is composed only of forward arcs, we infer that $F$ is a subdigraph of $\aG$. So $F$ can contain at most $\aell(\aG)$ edges, a contradiction. 

This finishes the proof that $\phi$ induces a monochromatic cycle in $D$.
\end{proof}

In the above proof, we note that if $\aG$ is an acyclic orientation of $G$, then the monochromatic cycle constructed is always of length $4$, and alternates between forward and backward arcs.

By combining Lemma~\ref{lem:LinearForest} and Lemma~\ref{lem:blowup}, we have the following result as a corollary.

\begin{corollary}
\label{cor:blow-up}
    Let $k$ be a positive integer, and let $G$ be an $n$-vertex graph such that $\chi(G) > k$. Then there exists an orientation $\aG$ of $G$ such that the $\big(k\big(n-\alpha(G)\big)+1\big)$-backward-blowup of $\aG$ is not $k$-dicolourable.
\end{corollary}

As a consequence, the $19$-backward-blowup of some orientation of the Grötzsch graph $M_4$ is not $3$-dicolourable. This yields a $4$-dichromatic triangle-free graph on $209$  vertices. We note that $\min\; \{\aell(G) : \mbox{$G$ is a triangle-free graph and $\chi(G) > 3$}\} = 6$, so the Grötzch graph is best possible when we apply Corollary~\ref{cor:blow-up} with $k=3$.
This holds because, given a triangle-free graph $G$ such that $\chi(G)>3$, and any independent set $I$ of $G$, one has $\chi(G\setminus I)>2$. If we assume for the sake of contradiction that $|V(G\setminus I)|\le 5$, then $G\setminus I$ is isomorphic to $C_5$. Therefore, for every vertex $x\in I$, the neighbourhood of $x$ is an independent set of $C_5$, so it has size at most $2$. We conclude that $G$ is $2$-degenerate, which contradicts the fact that $\chi(G) > 3$.

\begin{proposition}
    There exists a $4$-dichromatic $209$-vertex triangle-free graph. 
\end{proposition}

We note that because of a result by Mohar and Wu \cite[Lemma 3.2]{MoWu16}, Corollary~\ref{cor:blow-up} does not provide smallest-known $k$-dichromatic triangle-free graphs when $k\ge 5$. 

\begin{lemma}[Mohar, Wu, 2016]
    \label{lem:MoharWu}
    Let $k$ be a positive interger, and let $G$ be a graph such that $\chi(G)>k$. Let $m$ be an integer that satisfies $ 2 + 2 \ln m \le \ceil{m/k}$.
    Then there is an orientation of the $m$-blowup of $G$ that is not $k$-dicolourable.
\end{lemma}

Namely, if we want to apply Corollary~\ref{cor:blow-up} to the $r$-backward-blowup of an orientation of $M_5$ (the Mycielskian of the Grötzsch graph), we need $r=49$. In contrast, Lemma~\ref{lem:MoharWu} implies that there exists an orientation of the $37$-blowup of $M_5$ that is not $4$-dicolourable.

\subsection{Getting a better bound in the case \texorpdfstring{$k=2$}{k=2}}

We note that Lemma~\ref{lem:blowup} is tight when applied to any acyclic orientation of $C_5$ with $k=2$. In particular, for the two of them that contain no directed path of length $4$, the $7$-backward-blowup is not $2$-dicolourable, and this is best possible as illustrated in Figure~\ref{fig:6blowup_C5}.
However, it turns out that Lemma~\ref{lem:blowup} is no longer tight if we apply it to the directed $5$-cycle, and this is the orientation of $C_5$ that yields the minimum size of the backward-blowup that makes it not $2$-dicolourable.

\begin{figure}[ht]
        \centering
        \includegraphics[height=5cm]{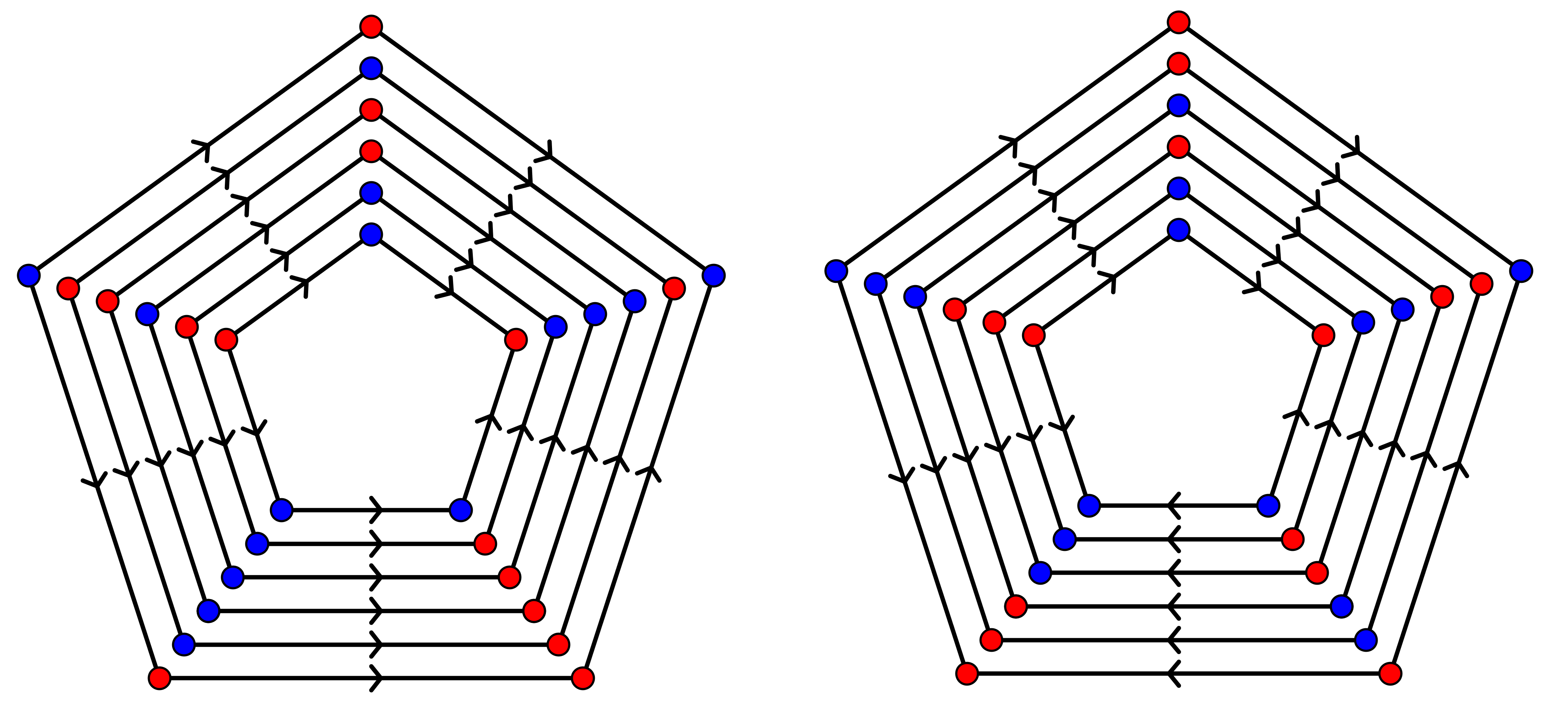}
        \caption{A 2-dicoloration of the 6-backward-blowup of the two orientations of $C_5$ that contain no directed path of length 4.}
        \label{fig:6blowup_C5}
    \end{figure}

In Proposition~\ref{prop:D25}, we show that the $5$-backward-blowup of $\overrightarrow{C_5}$ is not $2$-dicolourable, while Lemma~\ref{lem:blowup} only tells us that this holds for the $9$-backward-blowup of $\overrightarrow{C_5}$. More generally, it is possible to repeat the arguments used in the proof of Proposition~\ref{prop:D25} in order to show that the $\ell$-backward-blowup of $\overrightarrow{C_\ell}$ is not $2$-dicolourable, for every odd $\ell$.

\begin{proposition}
    \label{prop:D25}
    $D_{25}=\vec C_5^{\gets 5}$ is a $3$-dicritical oriented triangle-free graph on $25$ vertices
\end{proposition}

\begin{proof}
     Set $V(D_{25}) = \{u_{i, j} | i\in [5], j\in [5] \}$ and $A(D_{25}) = \{u_{i, j}u_{i+1, j'} \mid i\in [5], j\in [5], j'\in [5], j \neq j'\}\cup \{u_{i+1, j}u_{i, j} \mid i\in [5], j\in [5]\}$ reading the indices modulo 5.  The packs are $U_i=\{u_{i, j} | j\in [5] \}$ for $i\in [5]$.
    Two vertices are called \emph{matched} if the arc between them is forward.
    \begin{figure}[ht]
        \centering
        \includegraphics[height=5.5cm]{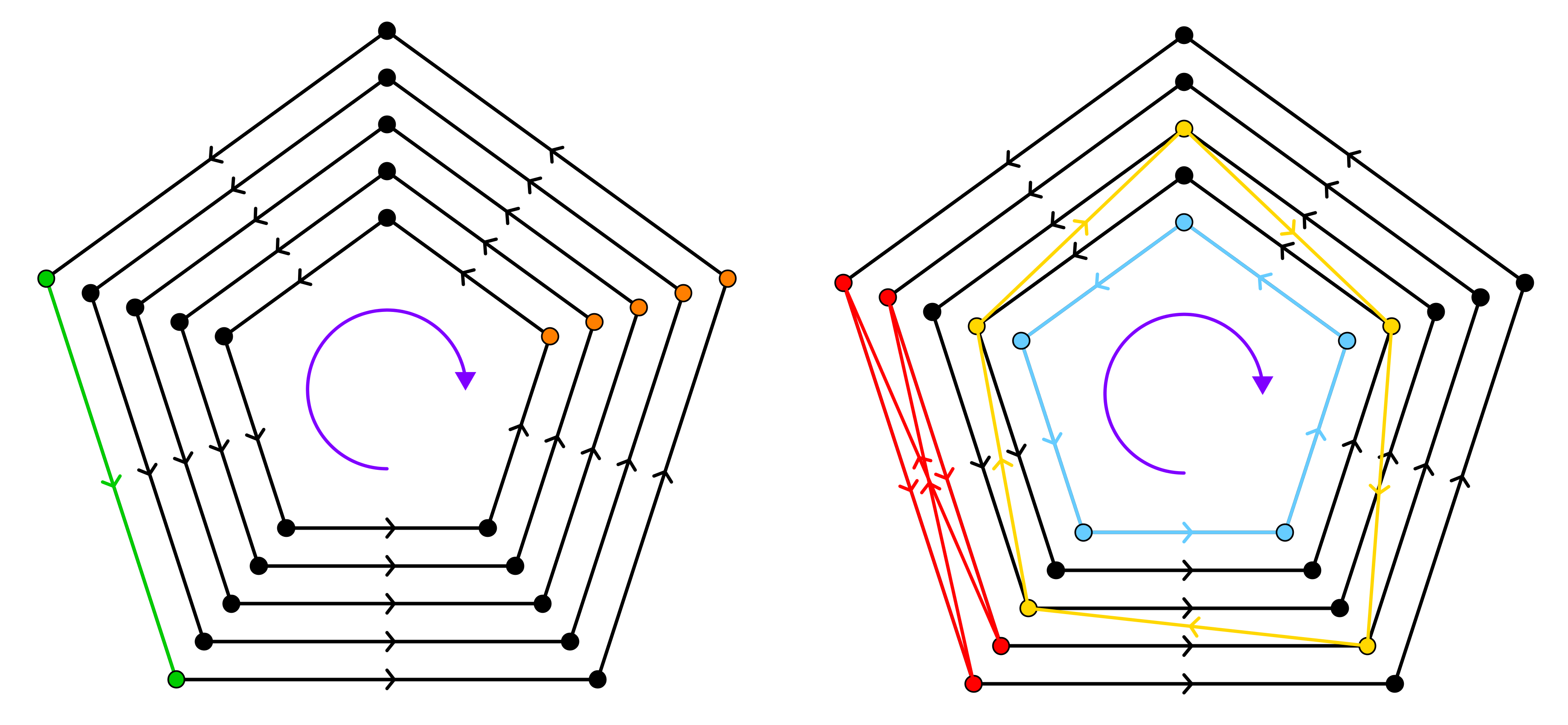}
        \caption{The digraph $D_{25}$. Left (only the forward arcs are represented): in green a pair of matched vertices, and in orange a pack. Right: the three main types of directed cycles.}
        \label{fig:5blowup_C5}
    \end{figure}
    Let us emphasize two types of directed cycles in $D_{25}$.
    \begin{itemize}
    \item \emph{matched cycle}: directed $4$-cycle induced by two pairs of matched vertices (for example the red arcs on Figure~\ref{fig:5blowup_C5} left), formally, $u_{i,j} \to u_{i+1,j} \to u_{i,j'} \to u_{i+1,j'} \to u_{i,j}$ for some $i\in [5]$ and two different indices $j,j'\in [5]$.
    \item \emph {backward cycle}: directed $5$-cycle with five backward arcs (for example the yellow arcs on Figure~\ref{fig:5blowup_C5} left), formally, $u_{1,j_1} \to u_{5,j_5} \to u_{4,j_4} \to u_{3,j_3} \to u_{2,j_2} \to u_{1,j_1}$  for some $j_1, \dots , j_5\in [5]$ such that $j_{i}\neq j_{i+1}$ for all $i\in [5]$. 
    \end{itemize}

    Assume that there is a 2-dicolouring $\phi$ of $D_{25}$.  
    
    First assume that there is a monochromatic pack. Without loss of generality, all vertices of $U_1$ are red. Since there is no monochromatic matched cycle,  both $U_2$ and $U_5$ have at most one red vertex and so at least $4$ blue vertices, and thus $U_3$ and $U_4$ have each at least least three red vertices. \\
    If both $U_2$ and $U_5$ have one red vertex, $u_{2,j_2}$ and $u_{5,j_5}$, then one can choose $j_1 \in [5]\setminus \{j_1, j_5\}$, $j_3\in [5]\setminus\{j_2\}$ and $j_4\in [5]\setminus \{j_3, j_5\}$ so that $u_{1,j_1}$, $u_{3,j_3}$, and $u_{4,j_4}$ are red. Then $u_{1,j_1} \to u_{5,j_5} \to u_{4,j_4} \to u_{3,j_3} \to u_{2,j_2} \to u_{1,j_1}$ is a monochromatic backward cycle, a contradiction. (See Figure~\ref{fig:D25_colo} left.)\\
    Henceforth, one of $U_2, U_5$, say $U_2$, has no red vertex and so five blue vertices.
   Since there is no blue matched cycle between $U_2$ and $U_3$, at least four vertices of $U_3$ are red. But then, since  there are at least three red vertices in $U_4$, there is a red matched cycle between $U_3$ ans $U_4$, a contradiction.
   Henceforth, we may assume that there is no monochromatic pack.
   
    \begin{figure}[ht]
        \centering
        \includegraphics[height=5.5cm]{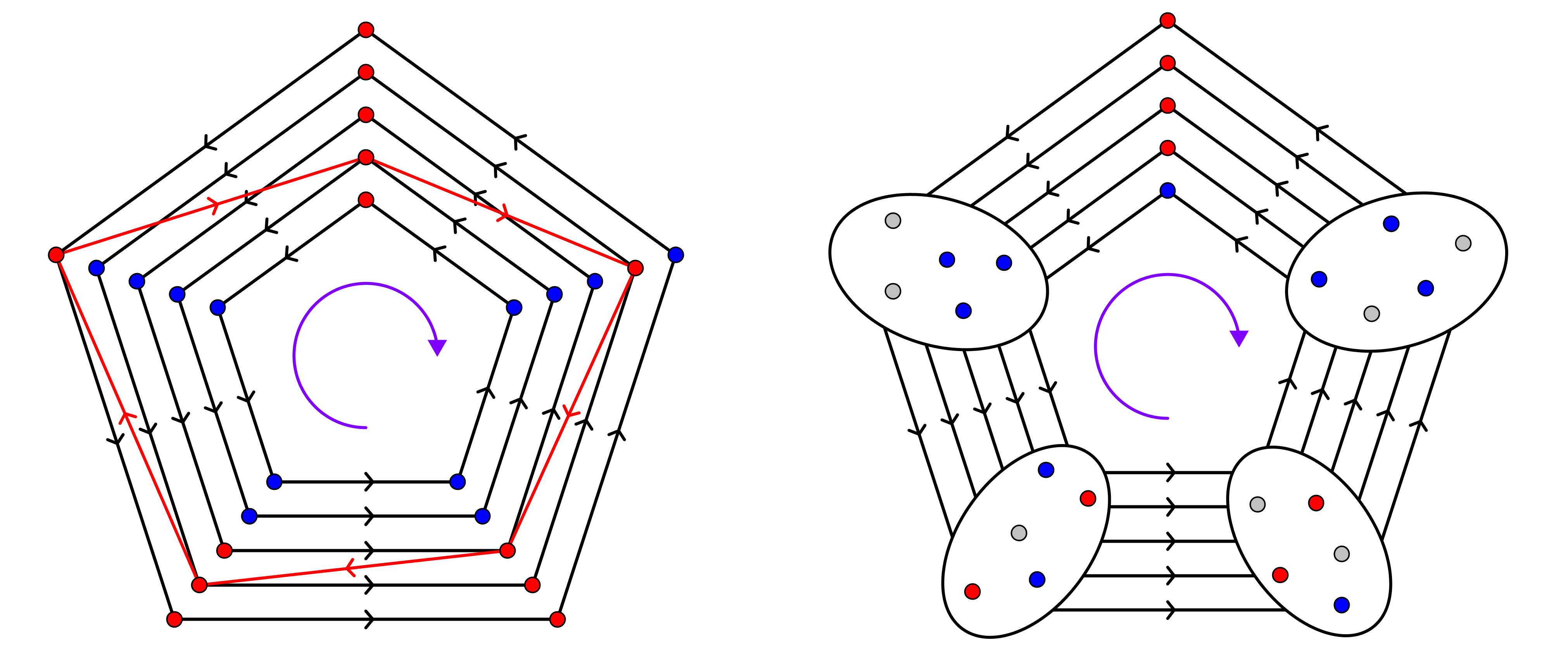}
        \caption{Left: a red backward cycle, when $U_1$ has five red vertices and $U_2$ and $U_5$ at least one.\\
        Right: The case with a pack with a $4$-$1$ colour partition.}
        \label{fig:D25_colo}
    \end{figure}

     Assume that there is a pack, say $U_1$, with four vertices of the same colour, say red. 
     Since there is no red matched cycles, each of $U_2, U_5$ has at least three blue vertices.
     Since there is no blue matched cycles, each of $U_3, U_4$ has at least two red vertices.
     Since no pack is monochromatic, there is a red vertex $u_{2,j_2}$ in $U_2$ and a red vertex $u_{5,j_5}$ in $U_5$.
     There exists $j_1\in [5]\setminus \{j_2, j_5\}$ such that $u_{1,j_1}$ is red. 
     There are no two distinct indices $j_3\in [5]\setminus \{j_2\}$ and $j_4\in [5]\setminus \{j_5\}$ such that $u_{3,j_3}$ and $u_{4,j_4}$ are red for otherwise there would be a red backward cycle.
     Then, necessarily, the two red vertices in $U_3$ are $u_{3,j_2}, u_{3,j_5}$ and the two red vertices in $U_4$ are $u_{4,j_2}, u_{4,j_5}$.
     But then those vertices induce a monochromatic matched cycle, a contradiction.
     
     Henceforth, each pack has two or three vertices in each colour. Without loss of generality,
     $U_1$ has three red vertices.
     Let $u_{2,j_2}$ be is a red vertex in $U_2$.
     Since there are at least two red vertices in each $U_i$, for $i=3$ to  $5$, we can find $j_i\neq j_{i-1}$ such that $u_{i,j_i}$ is red.
     Since $U_1$ has three red vertices, there is $j_1\notin \{j_2, j_5\}$ such that $u_{1,j_1}$ is red. Then $u_{1,j_1} \to u_{5,j_5} \to u_{4,j_4} \to u_{3,j_3} \to u_{2,j_2} \to u_{1,j_1}$ is a backward cycle, a contradiction.
     
     This proves that $D_{25}$ is not $2$-dicolourable.
    
    \medskip
    
    Figure~\ref{fig:D25_crit} depicts a 2-dicolouring of $D_{25}$ minus any vertex, $D_{25}$ minus a forward arc, and  $D_{25}$ minus a backward arc.
    Since all forward (resp. backward) arcs are equivalent, that is there is an automorphism of the graph mapping any forward (resp. backward) arc to any other forward (resp. backward) arc, this shows that $D_{25}$ is 3-dicritical.\\
    \begin{figure}[ht]
        \centering
        \includegraphics[height=4.5cm]{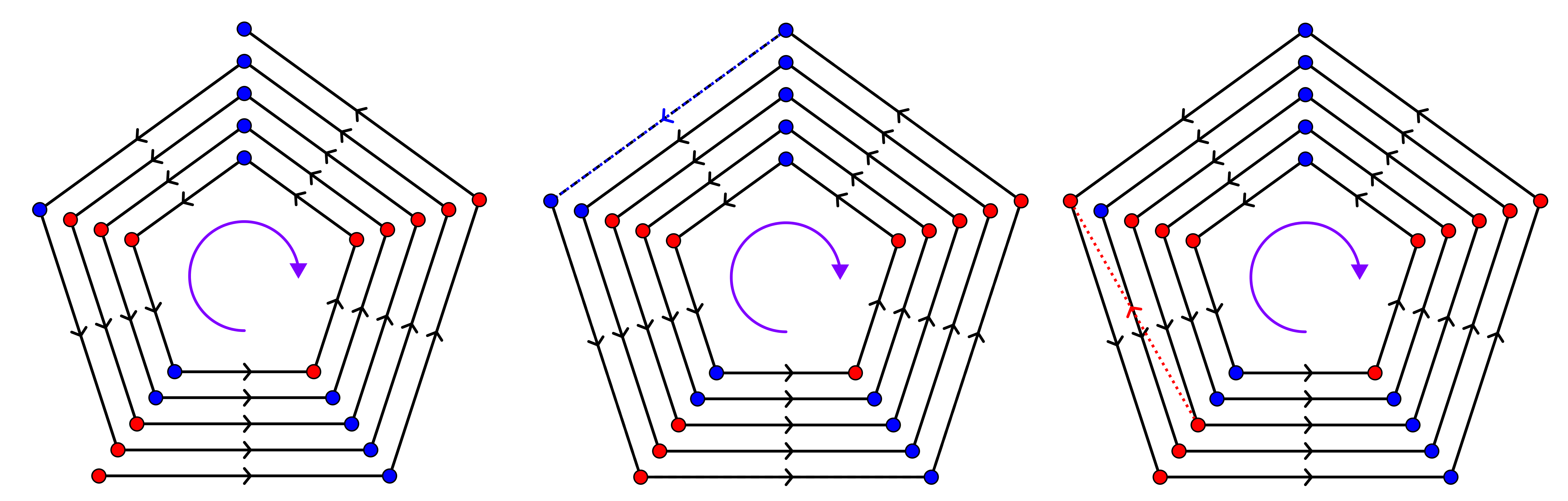}
        \caption{2-dicolouring of $D_{25}$ minus a vertex (left), $D_{25}$ minus a forward arc (dotted blue) (center), $D_{25}$ minus a backward arc (dotted red) (right).}
        \label{fig:D25_crit}
    \end{figure}
\end{proof}


\subsection{Lower bound on the number of vertices of oriented triangle-free graphs with dichromatic number 3}

\begin{lemma}
\label{lem:cut_2colo}
    Let $D$ be a oriented triangle-free graph. If $V(D)$ can be partitioned into three sets, $X$, $Y$ and $Z$ such that : 
    \begin{enumerate}[label=(\roman*)]
        \item $D[X]$ has no directed cycle,
        \item $Y$ is an independent set 
        \item there is no edge between $X$ and $Z$
        \item there is a $2$-dicolouring of $Z$ such that $c^{-1}(1)$ is either empty, a single vertex or two adjacent vertices,
    \end{enumerate}
    then $D$ is $2$-dicolourable.
\end{lemma}

\begin{proof}
    Consider the $2$-dicolouring $c$ of $Z$ given by $(iv)$ and extend it by colouring the vertices of $Y$ with colour 1 and the vertices of $X$ with colour 2. There is no monochromatic cycle in $Z$ nor in $X$, and none can intersect $X$ as its neighbourhood is included in $Y$, and therefore has a different colour. So if there were a monochromatic cycle, it would intersect both $Y$ and $Z$, thus have colour 1. $Y$ is pendent so the cycle contains two vertices of $Z$, but as $D$ is triangle-free, no vertex can be adjacent to both colour 1 vertices in $Z$, so there cannot be any monochromatic cycle.
\end{proof}


\begin{lemma}
\label{lem:digraph7}
For every triangle-free digraph $D$ on at most $7$ vertices, there exists a subset $X\subseteq V(D)$ which consists of one vertex or two adjacent vertices such that $D\setminus X$ is acyclic.
\end{lemma}

\begin{proof}
    Let $D$ be a triangle-free digraph  on at most 7 vertices of underlying graph $G$. We may assume that $\delta^-(D)\ge 1$ and $\delta^+(D) \ge 1$ (and so $\delta(G)\ge 2$), since a vertex of in- or out-degree $0$ is contained in no directed cycle.
    \begin{itemize}
        \item If $G$ is bipartite, by the minimum degree condition on $G$, either $G$ is isomorphic to $K_{2,5}$, or $G$ is a subgraph of $K_{3,4}$. 
        
        \smallskip
        In the first case, removing any of the two vertices of degree $5$ in $G$ yields an acyclic graph.
        
        \smallskip
        In the second case, let $(X,Y)$ be the bipartition of $G$, with $|X|\le 3$ and $|Y|\le 4$.
        If there is a vertex $x\in X$ with $d^-(x)=1$ or $d^+(x)=1$, let $u_x$ be the neighbour of $x$ through that unique ingoing or outgoing arc. Let $x'\in X$ be an other neighbour of $u_x$ in $G$. Then $D\setminus \{u_x,x'\}$ is acyclic, since there remains at most one vertex in $X$ of in- and out-degree at least $1$.  
        We can now assume that $d_G(x)=4$ for every $x\in X$, i.e. $G$ is isomorphic to $K_{3,4}$. For every vertex $y\in Y$, $d_G(y)= 3$, so either $d^+_D(y)=1$ or $d^-_D(y)=1$. Let $u_y$ be the neighbour of $y$ through that unique ingoing or outgoing arc. By the Pigeonhole Principle, there exists $x\in X$ such that $x=u_y=u_{y'}$ for two distinct vertices $y,y'\in Y$. Let $y''\in Y\setminus \{y,y'\}$. Then $D\setminus \{x,y''\}$ is acyclic, since there remains at most one vertex in $Y$ of in- and out-degree at least $1$.
        
        \item If $G$ is not bipartite, then it contains an odd cycle, which is either a $C_5$ or a $C_7$.
        
        \smallskip
        First, assume that $G$ contains a $C_5$ and no $C_7$. 
        Let $X$ be a set of vertices of $G$ that induces a $C_5$, and $Y\coloneqq V(G)\setminus X$. We have $|Y|\le 2$. The $C_5$ does not have any chord as it would form a triangle, so every other cycle of $G$ intersects $Y$.
        If $Y$ induces an independent set, then every vertex $y\in Y$ has exactly two neighbours in $X$ by the minimum degree condition on $G$. If there is a vertex $x$ of $X$ adjacent to all vertices of $Y$, then $D\setminus \{x\}$ is acyclic. Otherwise, $Y=\{y_1,y_2\}$, and there exist two consecutive vertices $x_1$ and $x_2$ on the $C_5$ such that $x_1\in N_G(y_1)$ and $x_2\in N_G(y_2)$. Then $D\setminus \{x_1,x_2\}$ is acyclic.
        If $Y=\{y_1,y_2\}$ induces an arc, then we cannot have two consecutive vertices $x_1,x_2$ on the $C_5$ such that $x_1\in N_G(y_1)$ and $x_2\in N_G(y_2)$, since this would yield a $C_7$. This implies that $y_1$ and $y_2$ have each exactly one neighbour in $X$, respectively $x_1$ and $x_2$. Then $D\setminus \{x_1,y_1\}$ is acyclic.
        
        \smallskip
        Assume now that $G$ contains a $C_7$. Since $G$ is triangle-free, it is a subgraph of the graph depicted in Figure~\ref{fig:C7} (when forgetting about the orientation). If the removal of $\{u_5,u_6\}$, of $\{u_6,u_0\}$, and of $\{u_0,u_1\}$ leaves a directed cycle, then this forces the partial orientation depicted in Figure~\ref{fig:C7}, or the reverse orientation. So removing $\{u_3\}$ leaves an acyclic digraph.
     \end{itemize}
     \end{proof}
        
        \begin{figure}[hbtp]
        \centering
        \begin{tikzpicture}
        \def\R{2}
        \foreach \i in {0,...,6} {
            \coordinate[fill, circle, scale=0.4] (x\i) at (360*\i/7-90:\R);
            \node at ((360*\i/7-90:1.2*\R) {$u_\i$};
        }
        \foreach \i in {0,...,5} {
            \pgfmathtruncatemacro{\j}{\i+1}
            \draw[thick,->] (x\i) -- (x\j);
        }
        \draw[thick] (x6) -- (x0);
        \draw[->] (x4) -- (x0);
        \draw[->] (x5) -- (x1);
        \draw[->] (x6) -- (x2);
        \end{tikzpicture}
        \caption{The case where $G$ contains a $C_7$}
        \label{fig:C7}
        \end{figure}
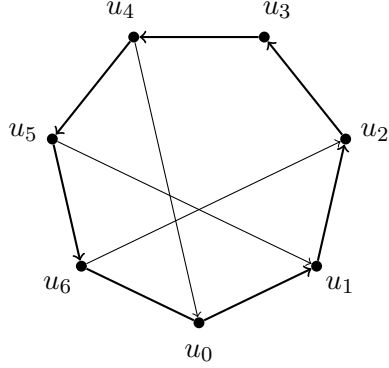

This lemma can be improved as follows.

\begin{lemma}
\label{lem:digraph8}
    Let $D$ be a oriented triangle-free graph on 8 vertices, then one of the following statements holds :
    \begin{enumerate}[label=(\roman*)]
        \item There is an arc $uv \in E(D)$ such that $D-\{u, v\}$ is acyclic.\label{it:2-acyclic}
        \item $D$ is an orientation of either two disjoint $4$-cycles, or the cube, or the cube with two diagonals, or $K_{4,4}$, or $K_{4,4}$ minus an edge, or the subgraph of $K_{4,4}$ with degree sequence $(2,2,4,4), (3,3,3,3)$.
        \label{it:exception}
    \end{enumerate}
\end{lemma}

This result has been obtained by enumerating all oriented triangle-free graphs on $8$ vertices and filtering out those that do not satisfy Property~\ref{it:2-acyclic} in Lemma~\ref{lem:digraph8}. More precisely, there are $83$ triangle-free graphs on $8$ vertices with minimal degree at least $2$ (one can generate them using nauty). Removing those that can be made acyclic by removing a pair of adjacent vertices, only $30$ remain.
Finally, by generating all possible orientations of these $30$ graphs, one can check that those that do not satisfy \ref{it:2-acyclic} satisfy \ref{it:exception} (there are $998$ such digraphs).

We note here that we were able to prove Lemma~\ref{lem:digraph8} by hand in the case where $D$ is not a subgraph of $K_{4,4}$. The method is similar to the proof of Lemma~\ref{lem:digraph7} but the proof is long and not particularly enlightening, so we choose not to include it here.

\begin{proposition}[Picasarri-Arrieta~\cite{Pic24}]
    If $\Delta(D) < 6$, then $D$ is $2$-dicolourable.
\end{proposition}

\begin{corollary}
    Every digraph on at most $14$ vertices is $2$-dicolourable.
\end{corollary}
\begin{proof}
    Take $u \in V$ with $d(u) \geq 6$. Then $X = \{u\}$ and $Y = N(u)$ satisfy the hypothesis of Lemma~\ref{lem:cut_2colo} since $|V \setminus (X \cup Y)| \leq 7$.
\end{proof}

This bound can be improved. Using Lemma~\ref{lem:digraph8}, we checked by computer calculus that every graph on at most $17$ vertices has a decomposition $(X,Y,Z)$ as described in Lemma~\ref{lem:cut_2colo}, and is therefore $2$-dicolourable, without having to consider their possible orientations. 
Using the fact that a $3$-dicritic graph must have minimum degree at least $4$ and arboricity at least $3$, we have proceeded as follows : 
We have enumerated all biconnected triangle-free graphs on $n\le 17$ vertices with minimum degree at least $4$ and maximum degree at most $n - 9$ (if $v$ has degree at least $n - 8$, then $X = \{v\}$, $Y = N(v)$ and $Z = V \setminus (X \cup Y)$ satisfy the hypothesis since $|Z| \leq 7$) using nauty. Then we filtered out all those of arboricity at most $2$.
For each graph $G$ obtained during this enumeration, we have computed several candidates for $(X,Y)$ by first fixing $X=\{u\}$ and $Y=N(u)$ for some vertex $v\in V(G)$, and adding to $X$ all vertices $u'\in V(G)$ whose neighbourhood is contained in $N(u)$ (since $G$ is triangle-free, $X$ and $Y$ are independent sets). We have kept the couple $(X,Y)$ for with $|X|+|Y|$ is maximised, and there remained to check whether $Z\coloneqq V(G) \setminus (X\cup Y)$ satisfied condition $(iv)$ of Lemma~\ref{lem:cut_2colo}.
It turns out that, when $n\le 17$, condition $(iv)$ is systematically satisfied for the set $Z \subset V(G)$ constructed as above because we have $\#\{z \in Z : \deg(z) \ge 2\} \le 8$, and when equality holds $G[Z]$ is not one of the exceptions listed in Lemma~\ref{lem:digraph8}.
More precisely for $n = 17$, we found $375$ graphs with arboricity at least $3$, among which $362$ led to a $Z$ with at most $7$ vertices, $12$ on with $8$ vertices and one with $10$ vertices. For all $Z$ with $8$ vertices, $G[Z]$ was not one of the exceptions listed in Lemma~\ref{lem:digraph8}. And the one on $10$ vertices has $2$ vertices with degree $1$, and the graph obtained by deleting them was not an exception.

\begin{lemma}
    Any $3$-dicritical oriented triangle-free graph has at least 18 vertices.
\end{lemma}

We stopped at $n=17$ only because of the time needed for the enumeration of triangle-free graphs --- using a laptop, the enumeration for $n=17$ took a week. Since we could find a decomposition with $X$ independent and $Z$ small for each graph, we believe that our lower bound is far from optimal. It might be possible to improve it with more computational power.

\section*{Acknowledgements}
This work was partially supported by the french Agence Nationale de la Recherche under contract Digraphs ANR-19-CE48-0013-01 and DAGDigDec (JCJC)   ANR-21-CE48-0012 and by the group Casino/ENS Chair on Algorithmics and Machine Learning.

\end{document}